\documentclass[12pt]{amsart}

\usepackage{graphicx, amsmath, amssymb, amsthm}

\usepackage{amsxtra}
\usepackage{amsopn}
\usepackage{amscd}
\usepackage{amsfonts}
\usepackage{latexsym}
\usepackage{verbatim}
\usepackage{enumerate}

\newcommand{\ad}{\mathrm{ad\,}}
\newcommand{\Ad}{{\mathrm {Ad\,}}}
\newcommand{\Aut}{\mathrm{Aut\,}}
\newcommand{\Der}{\mathrm{Der\,}}
\newcommand{\Dera}{\mathrm{Dera\,}}
\newcommand{\End}{\mathrm{End\,}}
\newcommand{\Ker}{\mathrm{Ker\,}}

\newcommand{\RR}{\mathbb R}
\newcommand{\CC}{\mathbb C}
\newcommand{\FF}{\mathbb F}
\newcommand{\NN}{\mathbb N}

\newcommand{\la}{\langle}
\newcommand{\ra}{\rangle}
\newcommand{\ct}{\mathrm{T}}

\newcommand{\fa}{{\mathfrak{a}}}

\newcommand{\fb}{{\mathfrak{b}}}

\newcommand{\fd}{{\mathfrak{d}}}

\newcommand{\frg}{{\mathfrak{g}}}

\newcommand{\fh}{{\mathfrak{h}}}

\newcommand{\fri}{{\mathfrak{i}}}

\newcommand{\fj}{{\mathfrak{j}}}

\newcommand{\fl}{{\mathfrak{l}}}

\newcommand{\fm}{{\mathfrak{m}}}

\newcommand{\fn}{{\mathfrak{n}}}

\newcommand{\fo}{{\mathfrak{o}}}

\newcommand{\fp}{{\mathfrak{p}}}

\newcommand{\fr}{{\mathfrak{r}}}

\newcommand{\fs}{{\mathfrak{s}}}

\newcommand{\fv}{{\mathfrak{v}}}

\newcommand{\fz}{{\mathfrak{z}}}

\newcommand{\fsl}{\mathfrak{sl}}

\theoremstyle{plain}

\newtheorem{theorem}{Theorem}[section]
\newtheorem{lemma}[theorem]{Lemma}
\newtheorem{proposition}[theorem]{Proposition}
\newtheorem{corollary}[theorem]{Corollary}
\newtheorem{example}[theorem]{Example}

\theoremstyle{definition}
\newtheorem{definition}[theorem]{Definition}

\theoremstyle{remark}
\newtheorem*{remark}{Remark}

\begin{document}

\title[Lie algebras with ad-invariant metrics \\ A Survey]
{Lie algebras with ad-invariant metrics \\ A Survey}

\author{Gabriela P. Ovando}

% select a language among: english, french, italian
%
%\selectlanguage{english}

% if you dont have footnote, cancel \footnotemark[1]
% separate several authors with ' - '
%

% text of the note (number from 1 to 9)
%
\footnotetext[1]{Partially supported by SeCyT - UNR and ANPCyT.}

\medspace

\begin{abstract}
This is a survey work on Lie algebras with ad-invariant metrics. We summarize main features, notions and constructions, in the  aim of bringing into consideration the main research on the topic.  We also give some list of examples in low dimensions. 
\end{abstract}

\maketitle

\hskip 190pt {\it To the memory of our friend Sergio}

\section{Introduction}
An ad-invariant metric on a Lie algebra $\frg$ is a non-degenerate  symmetric bilinear form $\la \,,\, \ra$ which satisfies
\begin{equation}\label{adme}
\la [x, y],  z\ra + \la y, [x, z]\ra = 0 \qquad \mbox{ for all }x, y, z \in  \frg.
\end{equation}

 Lie algebras equipped with  ad-invariant metrics are  known as:  ``metric''\cite{KO}, ``metrised'' \cite{Bor}, ``orthogonal'' \cite{MR2},  ``quadratic'' \cite{MR},  ``regular quadratic''\cite{FS}, ''symmetric self-dual''  \cite{FSt}, ''self-dual'' \cite{Pe} Lie algebras. Lie algebras which can be endowed with an ad-invariant metric are also known as ``quadrable'' or ``metrisable''. They  became relevant some years ago when they were useful in the formulation of some physical problems, for instance in the so known Adler-Kostant-Symes scheme. More recently they appeared in conformal field theory 
precisely as the Lie algebras for which a Sugawara construction exists \cite{FSt}.  They also constitute the basis for the construction of  bi-algebras and they give rise to interesting pseudo-Riemannian geometry \cite{Co}. For instance in \cite{Ov1} a result originally due to Kostant  \cite{Kos} was revalidated for pseudo-Riemannian metrics: it states that there exists a Lie group acting by isometries on  a pseudo-Riemannian naturally reductive  space (in particular symmetric spaces) whose Lie algebra can be endowed with an ad-invariant metric. For symmetric spaces this Lie group is called  the {\it transvection group} \cite{KO2}. 

Semisimple Lie algebras constitute examples of Lie algebras admitting an ad-invariant metric since the Killing form is non-degenerate.  In the solvable case, the Killing form is degenerate so one must search for another bilinear form with the ad-invariance property.  The family of Lie algebras admitting an ad-invariant metric  strictly contains semisimple Lie algebras and also infinitely many solvable examples, although not every solvable Lie algebra can be equipped with an ad-invariant metric.

As a survey work the aim here is  to bring into consideration  the research given until now in the topic. The presentation does not follow the chronology. We expose  basic definitions,  properties and examples. We give some structure results and classifications. The main constructions for these kind of Lie algebras are: the $T^*$-extension \cite{Bor}, the double extension procedure \cite{MR,FS} or the bi-extension \cite{HK}, the twofold or quadratic extension \cite{KO,KOl}. In the case of nilpotent Lie algebras a classification scheme was proposed in \cite{BCL}. Both the $T^*$ extension and the double extension procedure are good methods to produce examples but they present difficulties when dealing with the classification problem. The twofold or quadratic extension was given to offer an answer to this problem. In this survey we expose the main results and  refer to the original papers for the complete  proofs.  Indeed as we shall see the condition of having an ad-invariant metric imposes restrictions on the structure of the Lie algebra.  But even with all this information the following  question is still open: 

\smallskip

{\it Fix a Lie algebra $\frg$, does it admits an ad-invariant metric?}

\smallskip

Making use of the properties we shall see in the next sections that in many situations  we are able to give a negative answer to this question. In cases in which these conditions fail we are not able to decide if the fixed Lie algebra does or  does not admit an ad-invariant metric. However we have methods to construct examples and to deal with many questions, such as the study of the reach geometry that Lie algebras with ad-invariant metric induce,  as for instance pseudo-Riemannian symmetric spaces \cite{KO2}, compact homogeneous spaces, naturally reductive spaces, etc.  %This should be an application of the present work. 

\section{Generalities and structure theorems}
The aim of this section is to present the basic definitions and properties of Lie algebras with ad-invariant metrics. Together with the first question exposed in the Introduction we have the following one:
\smallskip

{\it  Which are the non-isomorphic Lie algebras admitting an ad-invariant metric?}

\smallskip

Indeed any advance concerning this question implies an advance in the previous one.   In the following we refer i-isomorphic or non-i-isomorphic Lie algebras in the following sense. This notion defines an equivalence among Lie algebras with ad-invariant metrics. 

\begin{definition} Two    Lie  algebras with ad-invariant metrics  $(\frg, B)$ and $(\frg', B')$   are  isometrically
isomorphic (or i-isomorphic, for short) if there exists a Lie algebra isomorphism $\psi: \frg \to \frg'$
satisfying $B'(\psi x, \psi y)=B(x,y)$ for all $x,y\in \frg$.   In this case, $\psi$ is called an i-isomorphism.
\end{definition}

The Killing form is an ad-invariant symmetric bilinear map on any  Lie algebra.  The semisimple Lie algebras are called {\it regular quadratic} since the Killing form is non-degenerate while  the Killing form has non-trivial kernel for solvable Lie algebras. 
Recall that the kernel of a symmetric bilinear map $B$ on the Lie algebra $\frg$  is the subspace $\Ker B\subset \frg$:
$$\Ker\,B=\{x\in \frg \,:\, B(x,y)=0\,\mbox{ for all }y\in \frg\},$$
so that $B$ is non-degenerate if and only if $\Ker \,B=\{0\}$.
Many examples of Lie algebras with ad-invariant metrics arise as cotangent Lie algebras. 

\begin{example} {\bf The coadjoint representation} Let $\fh$ denote a Lie algebra with dual space $\fh^*$. The coadjoint representation $\ad^*:\fh \to \End(\fh^*)$ is given by
\begin{equation}\label{coadjoint}
x\cdot \varphi (y) = \ad^*(x)\varphi(y)=-\varphi \circ \ad(x)(y)\qquad \mbox{ for all }\varphi\in \fh^*, x,y\in \fh.
\end{equation}

The {\em cotangent Lie algebra} is defined as the Lie algebra $\ct^*\fh=\fh \oplus\fh^*$, direct sum as vector spaces,  where the Lie bracket is given by
\begin{equation}\label{cot}
[(x_1, \varphi_1),(x_2, \varphi_2)]=([x_1, x_2]_{\fh}, x_1 \cdot \varphi_2 - x_2 \cdot \varphi_1) \qquad \mbox{ for all }x_1, x_2\in\fh, \, \varphi_1, \varphi_2\in \fh^*.
\end{equation}

 The natural neutral metric on $\ct^*\fh$ 
$$\la (x_1, \varphi_1), (x_2,\varphi_2)\ra= \varphi_1(x_2)+\varphi_2(x_1)$$ 
defines an ad-invariant metric on $\ct^*\fh$.   It is not hard to see that $\ct^*\fh$ is nilpotent if we start with a nilpotent Lie algebra  $\fh$. 
\end{example}

A generalization of the cotangent Lie algebra is given by the notion of $T^*$-extension of $\fh$. In fact the cotangent Lie algebra is obtained  as a $T^*$-extension for a trivial cocycle.

\begin{definition} Let $\fh$ denote a Lie algebra and let $\theta:\fh  \times \fh \to \fh^*$ denote a 2-cocycle of $(\fh, \ad^*)$. Let $\fh \oplus \fh^*$ with the canonical neutral metric and equipped with the Lie bracket given by
\begin{equation}\label{T-ext}
[(x_1, \varphi_1),(x_2, \varphi_2)]=([x_1, x_2]_{\fh}, x_1 \cdot \varphi_2 - x_2 \cdot \varphi_1 +\theta(x_1,x_2)) \quad \mbox{ for all }x_1, x_2\in\fh, \, \varphi_1, \varphi_2\in \fh^*.
\end{equation}
Then this Lie algebra denoted as $T^*_{\theta}\fh$ will be called the $T^*$-extension of $\fh$ by $\theta$. The canonical neutral metric on $\fh \oplus \fh^*$ is ad-invariant, once $\theta$ also satisfies
$$\theta(x_1,x_2)(x_3)=-\theta(x_1,x_3)(x_2) \qquad \mbox{for all}x_1, x_2, x_3\in \fh.$$ 
\end{definition}

The notion of $T^*$-extension in relation with ad-invariant metrics appeared in \cite{Bor}. It was used  by Bordemann to   get an example
of a Lie algebra of even dimension which is not a Manin triple.
On the other hand every Lie algebra with an ad-invariant metric can be shown to be a certain Manin pair in the sense of
 Drinfel’d. See \cite{Bor}. 

In  \cite{Bor} one can find the following features of $T^*$-extensions:

\begin{enumerate}[(a)]
\item if $\fh$ is solvable then $\ct^*_{\theta}\fh$ is solvable;
\item  if $\fh$ is nilpotent then $\ct^*_{\theta}\fh$ is nilpotent;
\item  non-isomorphic Lie algebras could have
isometric $T^*$-extensions. 
\end{enumerate}

\begin{example} This example shows the difficult in the last point above. In fact the 2 -step nilpotent Lie algebra  in three  generators $\fn_{3,2}$ can be obtained in two ways as  $T^*$-extensions, that is, starting from two non-isomorphic Lie algebras.  Let $\RR^3$ be the abelian Lie algebra. Consider the canonical basis $e_1, e_2, e_3$ with dual basis $e_4, e_5, e_6$ and the  2-cocycle $\theta:\RR^3 \times \RR^3 \to {\RR^3}^*$ given by
$$\theta(e_1,e_2)=e_6 \quad \theta(e_1,e_3)=-e_5\quad \theta(e_2,e_3)= e_4.$$
On the other hand $\fn_{3,2}$ can also be obtained as the cotangent Lie algebra of the Heisenberg Lie algebra $\ct^*\fh_3$.
\end{example}

\smallskip

Let $\frg$ denote a Lie algebra equipped with an ad-invariant metric $\la\,,\,\ra$.  
If $\fm \subseteq \frg$ is a subset, then we denote by $\fm^{\perp}$ the linear subspace of $\frg$, called the orthogonal of $\fm$ (relative to $\la\,,\,\ra$), given by
$$\fm^{\perp}=\{x\in \frg, \la x, v\ra=0 \mbox{ for all } v\in \fm\}.$$
 In particular a subspace $\fm$ is called 
\begin{itemize}
\item {\em isotropic} if $\fm \subseteq \fm^{\perp}$, 
\item {\em totally isotropic} if $\fm=\fm^{\perp}$, and 
\item {\em non-degenerate} if and only if $\fm \cap \fm^{\perp}=\{0\}$.
\end{itemize}

The pair consisting of a Lie algebra and an ad-invariant metric $(\frg,\la\,,\,\ra)$   is called {\em indecomposable} (irreducible in  \cite{BE}) if it does not admit a non-degenerate proper ideal. In the contrary situation $\frg$ is called decomposable.
The proof of the next result follows easily from  an inductive procedure.

\begin{lemma} \label{le1} Let $(\frg, \la\,,\,\ra)$ denote a Lie algebra equipped with an ad-invariant metric.
\begin{enumerate}[(i)]
\item If $\fh$ is an ideal of $\frg$ then $\fh^{\perp}$ is also an ideal in $\frg$.
\item $C^r(\frg)^{\perp}=C_r(\frg)$ for all $r$, 
where $C^r(\frg)$ and $C_r(\frg)$ are the ideals in the central ascending and descending series of $\frg$, defined inductively by
$$\begin{array}{rclrcl}
C^0(\frg) & = & \frg & C_0(\frg) & = & 0\\
C^r(\frg) & = & [\frg, C^r(\frg)] & \qquad C_r(\frg) & = & \{x\in \frg : [x,\frg]\in C_{r-1}(\frg)\}
\end{array}
$$
\end{enumerate}
\end{lemma}

Thus on any Lie algebra admitting an ad-invariant metric the next equality holds
\begin{equation}\label{e2}
\dim \frg=\dim C^r(\frg) + \dim C_r(\frg).
\end{equation}
If $C^1(\frg)$ is the commutator and $\fz(\frg)$ denotes the center of $\frg$, the case $r=1$ above gives
\begin{equation}\label{e1}
\dim \frg=\dim \fz(\frg) + \dim C^1(\frg),
\end{equation}
which easily implies that solvable Lie algebras admitting an ad-invariant metric have non-trivial center.

\begin{example} \label{4dim} Making use of the previous definitions and properties one can obtain the  Lie algebras admitting an ad-invariant metric in low dimensions. The Lie algebras of dimension less or equal four admitting an ad-invariant  metric are (see \cite{BOV} for the proof):
\begin{enumerate}[(i)]
\item In dimension = 1, 2  only the abelian Lie algebra. In fact the non-abelian solvable Lie algebra of dimension two admits no ad-invariant metric since it has a trivial center.

\item in dimension = 3, the simple Lie algebras  $\mathfrak{sl}(\RR)$ and $\mathfrak{so}(3)$ and the abelian Lie algebra. 

\item in dimension = 4, the trivial extension of three-dimensional simple Lie algebras, $\RR \times \mathfrak{sl}(2,\RR)$, $\RR \times \mathfrak{so}(3,\RR)$;  and the solvable Lie algebras: the oscillator Lie algebra $\mathfrak{osc}$ and another semidirect extension of the Heisenberg Lie algebra $\fb$ which corresponds to the known Boidol group, also called as the diamond Lie algebra. If one denotes by $\fh_3$ the Heisenberg Lie algebra of dimension three which is spanned by $\fh=span\{e_1, e_2, e_3\}$ with $[e_1, e_2]=e_3$ then 
\begin{itemize} 
\item $\mathfrak{osc}=span\{e_0,e_1,e_2,e_3\}$ and $[e_0, e_1]=e_2, \,[e_0, e_2]=-e_1,\,[e_1, e_2]=e_3$,
\item  $\fb=span\{e_0,e_1,e_2,e_3\}$ and $[e_0, e_1]=e_1, \,[e_0, e_2]=-e_2,\,[e_1, e_2]=e_3$.
\end{itemize}
\end{enumerate}
The proof follows the next schematic ideas. Assume $\frg$ is a Lie algebra of dimension four admitting an ad-invariant metric. Start  by considering the Levi decomposition of $\frg=\fr\oplus \fs$ where $\fr$ is the radical and $\fs$ is semisimple. If $\fs$ is non-trivial then it has dimension three and $\fr$ is decomposable. Otherwise $\fr$ is trivial and $\frg$ is solvable. In this situation assuming that $\frg$ is indecomposable one can see that such $\frg$ should satisfy $\dim\fz(\frg)=1$ and $\dim C^1(\frg)=3$. But $\fz(\frg)\subset C^1(\frg)$ and $C^1(\frg)/\fz(\frg)$ is isomorphic to a non-degenerate subspace of $C^1(\frg)$, see  the next remark to prove this. So $\frg=\RR \oplus C^1(\frg)$ is the Lie algebra of the semidirect product of $\RR$ and $C^1(\frg)$ is the Heisenberg Lie algebra  of dimension three. 
\end{example}

\begin{remark} Let $\frg$ denote a Lie algebra with an ad-invariant metric $\la\,,\,\ra$ and assume the center of $\frg$, denoted $\fz$,  is non-trivial. If $\fm\subset \fz$ is a subspace such that $\fz=\fm\oplus C$, where $C=C^1(\frg)\cap \fz$, then $\fm$ is non-degenerate %.% In fact  take $x\in \fm$ and assume $\la x, y\ra=0$ for all $y\in \fm$. Also $\la x, z\ra=0$ for all $z\in C$. Thus $x\in \fz^{\perp}=C^1(\frg)$ and therefore $x=0$. That is $\fm$ is a non-degenerate ideal, 
and so  $\frg$ is decomposable. 
\end{remark}

\begin{example} Let $\frg$ denote a $2$-step nilpotent Lie algebra equipped with an ad-invariant metric. Assume $\fz(\frg)=C^1(\frg)$, then by Equation (\ref{e1}) the metric is neutral and $\dim \frg = 2 \dim \fz(\frg)$.  As a consequence for instance the Heisenberg Lie algebra $\fh_n$ %and  nor the Lie algebra of Heisenberg type 
cannot  be equipped with an  ad-invariant metric. This example will generalized soon. 
\end{example}

Recall that a 2-step nilpotent Lie algebra $\fn$ is called {\em non-singular} if the maps $\ad_x:\fn \to \fz$ is non-singular for all $x\in \fn - \fz$, where $\fz$ denotes the center of $\fn$. 

Suppose $\fn=\fz\oplus \fv$, as direct sum as vector spaces.  Assume that $\fn$ is non-singular. Taking $x\in \fv$ it is not hard to see that $\dim \fv \geq \dim \fz +1$. Assume now $\fn$ admits an ad-invariant metric and it is indecomposable, so that this implies that $\fz=C^1(\fn)$. Making use of  Equation \ref{e1} above we have
$$\dim \fn= 2\dim \fz= \dim \fz \oplus \fv \geq 2\dim \fz + 1$$   
which gives a contradiction. This proves the next result.

\begin{corollary} Non-singular 2-step nilpotent Lie algebras cannot be endowed with an ad-invariant metric.  
\end{corollary}

Does a singular Lie algebra  admit an ad-invariant metric? The answer is no. One can see that the condition in Equation (\ref{e1}) (and hence Eq. (\ref{e2})) is not sufficient for a 2-step nilpotent Lie algebra to admit an ad-invariant metric as shown for instance in \cite{Ov2}.

Astrahancev in \cite{As}  studied conditions under which a metrizable Lie algebra $\frg$  can be decomposed into the orthogonal sum of non-degenerate ideals. The results  depend on the following main theorem: If $\frg$ is indecomposable (relative to an ad-invariant metric $B$), but $\frg$ is the sum of two commuting ideals, then one ideal must be $\frg$ itself and the other must be central. As a corollary, the indecomposability of $\frg$ is independent of the choice of $B$. 

In his work on superalgebras of low dimensions Duong gave some technical results  to find non i-isomorphic Lie algebras with ad-invariant metrics. One of the points to see is the indecomposability or irreducibility. For instance he gives a condition on the 2-cocycle $\theta$ to make  $\ct^*_{\theta}\frg$  decomposable. See \cite{Du}.

\begin{example} The $T^*$-extension does not exhaust the possibilities for constructing Lie algebras. In fact, in dimension four one gets the oscillator Lie algebra which is not the $T^*$-extension of the solvable Lie algebra of dimension two, which correspond to the diamond Lie algebra. 
\end{example}

 Bordemann proved the following result. 

\begin{proposition} \cite{Bor}
 Every finite-dimensional nilpotent 
Lie algebra of even dimension with an ad-invariant metric ``is'' a suitable $T^*$-extension. 
\end{proposition}

There is another way to produce Lie algebras with ad-invariant metric. This is an inductive method, that is, starting with a quadratic Lie algebra one produces a new one of a higher dimension. 
The first outline of this  construction %of Lie algebras with ad-invariant metrics 
was given in the book of V. Kac \cite{Kac} for solvable Lie algebras. Later it appeared as the double extension procedure in \cite{MR}. Also in \cite{FS}   the ideas of Kac are taken, as the authors  declare in the work.

Start with the following data:
\begin{itemize}
\item a Lie algebra $(\fd, [\cdot, \cdot]_{\fd})$ with an ad-invariant metric $\la\,,\,\ra_{\fd}$,
\item a Lie algebra $(\fh, [\cdot, \cdot]_{\fh})$ with ad-invariant symmetric bilinear (possibly degenerate) form $B_{\fh}$,
\item a Lie algebra homomorphism $\pi : (\fh, [\cdot , \cdot ]) \to \Dera(\fd, \la\, ,\,\ra_{\fd})$ from $\fh$ to the Lie algebra of skew-symmetric derivations of $\fd$, denoted by $\Dera(\fd, \la\, ,\,\ra_{\fd})$.

\end{itemize}
Consider the following vector space direct sum
$$\frg := \fh \oplus \fd \oplus \fh^*.$$
Let $Q$ be the symmetric bilinear map on $\frg$, which for $x_i\in \fd, \alpha_i \in\fh^*, h_i \in  \fh, i = 1, 2,$ is defined by
\begin{equation}\label{Q} 
Q((h_1, x_1, \alpha_1), (h_2, x_2, \alpha_2)) := B_{\fh}(h_1, h_2) + \la x_1, x_2\ra_{\fd} + \alpha_1(h_2) + \alpha_2(h_1);
\end{equation}
it is non-degenerate and of signature $sgn(Q) = sgn(\la\, ,\, \ra_{\fd}) + (\dim \fh, \dim \fh)$. 

Let $\ad_{\fh}$ denote the adjoint action of $\fh$ to itself. The Lie bracket on $\frg$ is given by
\begin{equation}\label{bra}
\begin{array}{rcl}
[(h_1, x_1, \alpha_1), (h_2, x_2, \alpha_2)]
& := &  ([h_1, h_2]_{\fh}, [x_1, x_2]_{\fd} + \pi(h_1)x_2 - \pi(h_2)x_1, \\
& & \qquad \beta(x_1, x_2) + ad^*_{\fh} (h_1) \alpha_2 - \ad^*_{\fh} (h_2) \alpha_1)
\end{array}
\end{equation}

where $\beta(x_1, x_2)(h) := \la \pi(h)x_1, x_2\ra_{\fd}$ and $\ad^*_{\fh}:\fh \to \End(\fh^*)$ denotes the coadjoint action. 
 This formula and usual computations show that
\begin{itemize}
\item  the metric $Q$ is ad-invariant with respect to this bracket.
\item While $\fh$ is a Lie subalgebra of $\frg$, in general $\fd$ is not a subalgebra. 
\item The subspace $\mathcal G(\fd) := \fd \oplus \fh^*$ is always an ideal in $\frg$, which
is obtained as a central extension of $\fd$ by the 2-cocycle $\beta$ and $\frg=\fh \oplus \mathcal G(\fd)$, said the semidirect sum of $\mathcal G(\fd)$ and $\fh$.
\end{itemize}

The resulting Lie algebra $\frg$ is called the {\em double extension} or {\em bi-extension} of $\fd$ with respect to $(\fh, \pi)$. One also says that $\frg$ is obtained by the double extension procedure or double extension process.

Notice that
\begin{itemize}
\item The commutator of $\frg$ consists of $$C(\frg)\subseteq C(\fh)+ C(\fd)+Im \pi + Im \beta + Im \ad^*_{\fh}.$$
In particular if $\dim \fh=1$ then $C(\frg)\subseteq   C(\fd)+Im \pi + Im \beta.$
%\item The signature of $Q$ is $sign(Q)=sign \la\,,\,\ra_{\fd}+(m,m)$ where $m=\dim \fh$. 
\item If $\frg$ is solvable, then both $\fh$ and $\mathcal G(\fd)$ are solvable. 
\item If $\frg$ is nilpotent then both $\fh$ and $\mathcal G(\fd)$ are nilpotent. Also $\pi$ is acting by nilpotent maps.
\end{itemize}

\begin{example} In the double extension process starting with  $\fd=0$ one gets the cotangent Lie algebra $\fh\oplus \fh^*$ in (\ref{cot}).
\end{example}

\begin{example} Let  start the double extension process with $\fd=\RR^2$. If we equip  $\RR^2$ with the canonical metric $\la \,,\,\ra$ and take  the map $J$ corresponding to the canonical complex structure, which is skew-symmetric for $\la \,,\,\ra$, we get the 2-form $\beta(x,y)=\la Jx,y\ra$. The Lie algebra $\frg$ obtained after the double extension procedure is  the {oscillator Lie algebra} denoted $\mathfrak{osc}$ in Example \ref{4dim}. It is solvable and it has a Lorentzian ad-invariant metric.

If we equip  $\RR^2$ with the neutral metric and follow the double extension procedure we get the Lie algebra $\fb$ of Example \ref{4dim}, whose ad-invariant metric is neutral. 
\end{example}

%\begin{remark} \cite{Du} Let $(\frg, B)$ be a  Lie algebra  equipped with an ad-invariant metric and $C = \ad(X_0)$ be an inner derivation of $\frg$. Then the double extension  of $\frg$ by $C$ is decomposable. 
%\end{remark}

The following results describe the structure of Lie algebras with ad-invariant metrics. The first one  assumes that the Lie algebra is indecomposable. The second one does not take this assumption but it needs the existence of a non-trivial center.

\begin{theorem} \label{MR} {\sc Medina and Revoy '85.}
A non-simple Lie algebra $(A,Q)$ equipped with an ad-invariant metric which is indecomposable is a double extension of some Lie algebra $(\fd, \la\,,\,\ra_{\fd})$ by a one-dimensional Lie algebra or by a simple Lie algebra. 
\end{theorem}

\begin{proof} See  \cite{MR}
\end{proof}

\begin{theorem} \label{FS} {\sc Favre-Santharoubane '87.}  Any n-dimensional  Lie algebra with an ad-invariant metric and 
non-trivial center (in particular an n-dimensional  solvable
Lie algebra with ad-invariant metric) is either an orthogonal direct sum of (n - 1 )-dimensional and
l-dimensional  Lie algebras with ad-invariant metrics (if it contains an anisotropic
central element) or is a double extension of an (n - 2)-dimensional Lie algebra Lie algebra with ad-invariant metric $(\fd, \la\,,\,\ra_{\fd})$.
\end{theorem}

\begin{proof}
See \cite{FS}.
\end{proof}

Apparently  in her Dissertation \cite{Ke}, V. Keith got a similar procedure for the construction of quadratic Lie algebras. This is announced in the work by Hofmann and Keith \cite{HK}. They make use of the notion of {\em bi-extension} to name this construction. With the years the method was known by the {\it double extension procedure}. However it is interesting to note that the tools and techniques used by Hofmann and Keith are the same as in \cite{MR}.

\begin{proposition} \cite{FSt} In the context of the double extension procedure above so that the Lie algebra $\frg$ is a double extension of $\fd$ via ($\fh, \pi)$. 
\begin{enumerate}[(i)]
\item If $\pi \equiv 0$, then the Lie algebra $\frg$ decomposes into an orthogonal direct sum of the Lie algebras
$\fd$ and the cotangent Lie algebra $\fh \oplus \fh^*$. 
\item  If $\pi \to \Dera(\fd, \la \,,\,\ra)$  is given by inner derivations, i.e., if there is a homomorphism $\varphi: \fh \to \fd$
such that
$\pi(H)= \ad_{\fd} \varphi(H)$,
then there exists an isometric Lie algebra isomorphism
$\Phi: \frg \to \fd \times (\fh \oplus \fh^*)$,
where the metric on the image is given by the product of the metrics $\la \,,\,\ra_{\fd}$ and the metric on $\fh \oplus \fh^*$ given by
$$\la (\alpha_1,H_1), (\alpha_2, H_2)\ra= \la H_1, H_2\ra_{\fh} + \la \varphi(H_1), \varphi(H_2)\ra_{\fd} + \alpha_1(H_2) + \alpha_2(H_1).$$

\item If $\fd$ has a factor $\fp$ with $H^1(\fp,\RR) = H^2(\fp,\RR) = 0$, then its double extension $\frg$ is
decomposable.
\end{enumerate}
\end{proposition}

Favre and Santharoubane investigated isomorphism classes of double extensions of a given Lie algebra by a skew-symmetric derivation.

\begin{proposition} \label{isom} \cite{FS} Let $(\frg_d, \la\,,\,\ra_d)$ denote a one-dimensional extension of the Lie algebra $\fd$ with ad-invariant metric by the derivation $d$.  There is an isometry $\alpha: (\frg_d, \la\,,\,\ra_d) \to (\frg_t, \la\,,\,\ra_t)$  such that $\alpha(\fd \oplus \RR^*) = \fd \oplus \RR^*$ if and only if  
if and only if there exists $\lambda_0 \in \RR-\{0\}$, $x_0\in \fd$ and $\varphi \in \Aut(\fd, \la\,,\,\ra_{\fd})$ such that
$\varphi^{-1} t \varphi= \lambda_0 d + \ad(x_0)$.
\end{proposition}

%\begin{remark} Let  $\fd$ be the abelian Lie algebra equipped with a metric of signature $(p,q)$. An application of Proposition \ref{isom} says that for $A$, $B$ skew-symmetric maps, the extensions $\frg_A$ and $\frg_B$ are isomorphic if and only if there exists $T$ a non-singular map on $\fd$ and $\lambda_0\in \RR-\{0\}$ such that $T^{-1} B T = \lambda_0 A$.
%\end{remark}

Let $\fz^1(\fd):=\Dera(\fd, \la\,,\,\ra)$ denote the set of skew-symmetric derivations of $\fd$. It is clear that $B^1(\fd):=\ad(\fd)\subseteq \fz^1(\fd)$. Set $H^1(\fd):=\frac{\fz^1(\fd)}{B^1(\fd)}$ and let $\mathbb P^1 H^1(\fd)$ denote the projective  space. Then the group $Auto(\fd)$ of orthogonal automorphisms of $(\fd,\la\,,\,\ra)$  operates on $\mathbb P^1 H^1(\fd)$ by
$$\alpha \cdot K [d] = K[ \alpha^{-1} d \alpha]$$
where $[d]$ denotes the image of $d$ by the projection $\Dera(\fd) \to H^1(\fd)$.

The action of $Aut(\fd)$ in $H^1(\fd)$ is identical to the adjoint
action of the Lie group $O(n)$ in its Lie algebra $\fo(n)$. Therefore one needs to determine the $O(n)$-orbits in $\fo(n)$. For the
nilpotent orbits the solution is classical: for any $d\in \fo(n)$ the $O(n)$-orbit
of $d$  equals the intersection with $\fo(n)$ of the $GL(\fd)$-orbit of $d$. This
implies that the $O(n)$-orbits are determined by the partition $(n_1, \hdots , n_t)$ of $n$
satisfying an extra  condition. See \cite{FS}.

The next method  is known as {\it twofold extension}. As declared in \cite{KO} it was used by Berard Bergery to study  symmetric spaces
(in up to now unpublished work). Kath and Olbrich applied this construction to produce  metric Lie algebras. They show 
that each metric Lie algebra belonging to the class of solvable metric Lie algebras is a twofold extension associated
with an orthogonal representation of an abelian Lie algebra. They  describe
equivalence classes of such extensions by a certain cohomology set. Next we summarize this construction and interested results in relation with this. 

Let $(\rho, \fa)$ be an orthogonal representation of an abelian Lie
algebra $\fl$ on the semi-Euclidean vector space $(\fa, \la \cdot  , \cdot \ra_{\fa})$. Furthermore, choose a 3-form $\gamma \in \Lambda^3 \fl^*$ and a cocycle $\alpha \in Z^2(\fl, \fa)$ satisfying
\begin{equation}\label{alfa}
\la \alpha(L_1,L_2), \alpha(L_3,L_4)\ra_{\fa} + \la \alpha(L_2,L_3), \alpha(L_1,L_4)\ra_{\fa} + \la \alpha(L_3,L_1), \alpha(L_2,L_4)\ra_{\fa} = 0
\end{equation}
for all $L_1,L_2,L_3,L_4 \in \fl$. Then the bilinear map $[ \cdot , \cdot ] : (\fl^* \oplus \fa \oplus \fl)^2 \to \fl^* \oplus \fa \oplus \fl$
defined by $\fl^* \subset \fz(\fl^* \oplus \fa \oplus \fl)$ and
$$
\begin{array}{rcl}
{[A_1,A_2]} &  =  &\la \rho(\cdot)A_1,A_2\ra_{\fa}\in \fl^* \\
{[A,L]}  & = &\la A, \alpha(L, \cdot )\ra_{\fa} - L(A) \in \fl^* + \fa\\
{[L_1,L_2]} & = & \gamma (L_1,L_2, \cdot ) + \alpha(L_1,L_2) \in \fl^* + \fa
\end{array}
$$
for all $L,L_1,L_2 \in \fl$, $A,A_1,A_2 \in \fa$ is a Lie bracket on $\fl^*\oplus \fa \oplus \fl$ and the symmetric bilinear
form $\la \cdot , \cdot \ra$ on $\fl^* \oplus \fa \oplus \fl$ defined by
$$\la Z_1 + A_1 + L_1,Z_2 + A_2 + L_2\ra = \la A_1,A_2\ra_{\fa} + Z_1(L_2) + Z_2(L_1)$$
for all $Z_1,Z_2 \in \fl^*$, $A_1,A_2 \in \fa$ and $L_1,L_2 \in \fl$ is an ad-invariant metric (with respect to the bracket above)  on $\fl^* \oplus \fa \oplus \fl$. The Lie algebra $(\fl^* \oplus \fa \oplus \fl, \la\cdot, \cdot \ra)$ is denoted by $ \fd_{\alpha, \gamma}(\fa, \fl,\rho)$ an it is called a {\it twofold extension}.

\begin{proposition}\cite{KO} If $(\frg, \la \cdot , \cdot \ra)$ is a non-abelian indecomposable metric Lie algebra such that 
$\frg' / \fz(\frg)$ is abelian then there exist an abelian Lie algebra $\fl$, a semi-Euclidean vector space $(\fa, \la \cdot , \cdot \ra_{\fa})$, 
an orthogonal representation $\rho$ of $\fl$ on $\fa$, a cocycle $\alpha\in Z^2(\fl, \fa)$ satisfying (\ref{alfa}), and
a 3-form $\gamma\in \Lambda^3\fl^*$ such that $(\frg, \la \cdot , \cdot \ra)$ is isomorphic to $\fd_{\alpha, \gamma}(\fa, \fl,\rho)$. In particular one can choose $\fa = \frg'/\fz(\frg)$, $\fl = \frg/\frg'$. 
\end{proposition}

It is clear that for an indecomposable solvable metric Lie algebra one has that $\frg'/ \fz(\frg)$ is abelian. Conversely Kath and Olbrich proved that any  non-simple indecomposable metric Lie algebra $\frg$ of index
at most 2 satisfies $\frg'/ \fz(\frg)$ is abelian. And in this case $\frg$ is solvable. Kath and Olbrich in \cite{KO} state the following structure theorem. 

\begin{theorem}\cite{KO}
 If $(\frg, \la \cdot, \cdot \ra)$ is an indecomposable metric Lie algebra of signature
$(p, q)$, where $p \leq  q$, with non-trivial maximal isotropic centre, then there exist an abelian
Lie algebra $\fl$, a Euclidean vector space $(\fa, \la \cdot, \cdot \ra_{\fa})$, an orthogonal representation $\rho$ of $\fl$ on $\fa$, a cocycle $\alpha \in Z^2(\fl, \fa)$ satisfying (\ref{alfa}), and a 3-form $\gamma \in \Lambda^3 \fl^*$ such that
$\fd_{\alpha, \gamma} (\fa, \fl, \rho)$ is regular and $(\frg, \la \cdot , \cdot \ra)$ is isomorphic to $\fd_{\alpha, \gamma} (\fa, \fl, \rho)$.
\end{theorem} 

The Lie algebra $\fd:= \fd_{\alpha, \gamma} (\fa, \fl, \rho)$ is called {\em regular} if $\fz(\fd)=\fl^*$. Also an equivalence condition among twofold extensions was found. 

\begin{proposition} \cite{KO} The metric Lie algebras $\fd_{\alpha_i ,\gamma_i} (\fa, \fl, \rho)$ are extension equivalent if
and only if there exists a cochain $\tau \in C^1(\fl, \fa)$ such that $\alpha_2 - \alpha_1 = d\tau$ and
$\gamma_2 - \gamma_1 = \la(\alpha_1 + \frac12 d\tau) \wedge \tau \ra_{\fa}$. 
\end{proposition}
 Above one has
$$\la \alpha \wedge \tau\ra_{\fa}(L_1,L_2,L_3) = \la \alpha(L_1,L_2), \tau(L_3)\ra_{\fa} + \la \alpha(L_3,L_1), \tau (L_2)\ra_{\fa}
+\la \alpha(L_2,L_3), \tau(L_1)\ra_{\fa}.$$

Take $\mathcal H^2_
C(\fl, \fa) := (Z^2_C(\fl, \fa) \times \Lambda^3 \fl^*)/C^1(\fl, \fa)$ where $Z^2
_C(\fl, \fa)\subset  Z^2(\fl, \fa)$ denotes the space of all cocycles satisfying Equation (\ref{alfa}). Thus fix $(\fa, \fl, \rho)$. The correspondence
$\fd_{\alpha, \gamma}(\fa, \fl, \rho) \to  [\alpha,\gamma] \in \mathcal H^2_C(\fl, \fa)$
defines a bijection between the extension equivalence classes of metric Lie algebras
of the form $\fd_{\alpha,\gamma} (\fa, \fl, \rho)$ and elements of $\mathcal H^2_C(\fl, \fa)$. 
Also some  decomposability conditions needed for classifications results were given. 

In \cite{KOl} the same authors are able to carry over the approach of \cite{KO} to general metric
Lie algebras.

The idea is the following: Let
$(\frg, \la\cdot, \cdot\ra)$ denote a Lie algebra equipped with an ad-invariant metric. Assume that there is an isotropic ideal
$\fri\subset \frg$ such that $\fri^{\perp}/\fri$ is abelian. Set
$\fl:=\frg/\fri^{\perp}$ and $\fa=\fri^{\perp}/\fri$.  Then $\fa$
inherits an inner product from $\frg$ and an $\fl$-action respecting this
inner product, i.e., it inherits the structure of an orthogonal $\fl$-module. Moreover, $\fri\simeq \fl^*$ as an $\fl$-module, and
$\frg$ 
can be represented as the result of two subsequent extensions of
Lie algebras with abelian kernel:
\begin{equation}\label{extension}
0 \longrightarrow \fa \longrightarrow \frg/\fri \longrightarrow \fl \longrightarrow 0, \qquad 0 \longrightarrow \fl^* \longrightarrow
\frg \longrightarrow \frg/\fri \longrightarrow 0.
\end{equation}

Vice versa, starting with the elements above satisfying some compatibility conditions one can construct a metric Lie algebra. This is the idea under the double extension procedure which can be formalised by the quadratic extension of $\fl$ by an orthogonal
$\fl$-module $\fa$. 

The cocycles defining the extensions in the sequences (\ref{extension}) represent an element in a certain cohomology set $\mathcal H^2_Q(\fl,\fa)$, and it turns out that there is a bijection between equivalence classes of such quadratic extensions and $\mathcal H^2_Q(\fl,\fa)$.

What makes the theory of quadratic extensions  useful is the fact that any metric Lie algebra $(\frg, \la\,,\,\ra)$ without simple ideals has a canonical isotropic ideal $\fri$ such that
$\fri^{\perp}/\fri$ is abelian. In other words, $(\frg,\la\cdot,\cdot\ra)$ has a canonical structure of a quadratic extension of
$\fl=\frg/\fri^{\perp}$ by $\fa=\fri^{\perp}/\fri$.  However not every quadratic extension of a Lie algebra $\fl$ by a semi-simple orthogonal $\fl$-module $\fa$ arises in this way. The obvious condition that the image of $\fl^*$ is $\fri\subset \frg$ in Equation (\ref{extension}) is not always satisfied. If it is satisfied one calls the quadratic extension {\em balanced} and the corresponding cohomology class in
$\mathcal H^2_Q(\fl, \fa)$ {\em admissible}.  Denote the set of
indecomposable admissible cohomology classes by
$\mathcal H^2_Q(\fl, \fa)_0$. It turns out that elements of
$\mathcal H^2_Q(\fl, \fa)_0$
correspond to isomorphic Lie algebras if and only if they can
be transformed into each other by the induced action of the automorphisms group $G_{\fl,\fa}$ of the pair $(\fl,\fa)$. This was the work of Kath and Olbricht. 

Another approach was recently proposed for quadratic nilpotent Lie algebras. In  \cite{BCL} Benito, de-la-Concepci\'on and Laiena develop a general classification scheme for quadratic nilpotent Lie algebras based on the use of invariant bilinear forms on free nilpotent Lie algebras. The main ideas are summarized in the next paragraphs.

Let $\fn$ denote a nilpotent Lie algebra. The {\em type} of $\fn$
is defined as
 the codimension of $C^1(\fn)$ in $\fn$. That is {\em type} of $\fn = \dim \fn -\dim C^1(\fn)$. From this it is clear that the type of $\fn$ is the cardinal of a minimal set of generators of $\fn$. 

Let $\fn_{d,t}$ denote the free t-step nilpotent Lie algebra in $d$ generators. Following \cite{Ga}, any $t$-step nilpotent Lie algebra
$\fn$ of type $d$ is a homomorphic image of $\fn_{d,t}$. Thus $\fn=\fn_{d,t}/I$ for an ideal  $I$  of $\fn_{d,t}$. 

Assume $(\fn, B)$ is a quadratic $t$-step nilpotent Lie algebra of type
$d$ and $\varphi:\fn_{d,t}/I \to \fn$ is  an isomorphism of Lie algebras. We can define on
$\fn_{d,t}$ 
the following ad-invariant symmetric
bilinear form:
$$B_1(x,y)=B(\varphi(x+I), \varphi(y+I)).$$

\begin{proposition} Let
$(\fn,B)$ be a quadratic $t$- step nilpotent Lie algebra of type
$d$ and let $\varphi: \fn_{d,t}/I \to \fn$
be an isomorphism of Lie algebras. Then:
\begin{enumerate}[(i)]
\item The map
$B_1:\fn_{d,t}\times\fn_{d,t} \to K$ above is an
invariant symmetric bilinear form on $\fn_{d,t}$.
\item One has  $\Ker(B_1)=\fn_{d,t}^{\perp}=I$ 
and it satisfies that $I\subseteq C^1(\fn_{d,t})$. 
\item The map $\bar{B}_1:\fn_{d,t}/I\times\fn_{d,t}/I \to K$ given by $\bar{B}_1(x+I,y+I)=B_1(x,y)$ is an
ad-invariant non-degenerate symmetric bilinear form on
$\fn_{d,t}/I$. 
\item $\varphi$ 
is an isometry from
$(\fn_{d,t}/I, \bar{B}_1)$ onto $(\fn, B)$. 
\end{enumerate}
\end{proposition}

From this result it is clear that the classification of quadratic nilpotent Lie algebras is related to free nilpotent Lie algebras. 
The authors define $\mathrm{NilpQuad}_{d,t}$ as the category whose objects are the quadratic $t$-step nilpotent 
Lie algebras $(\fn,B)$  of type $d$, and whose morphisms are Lie homomorphism  $\varphi: (\fn, B)\to (\fn', B')$ preserving the bilinear forms, that is $B'(\varphi x, \varphi y)=B(x,y)$. These are called {\em metric Lie homomorphisms}. Define also
$\mathrm{Sym}_0(d,t)$ the category whose objects are the symmetric invariant bilinear forms
$B$ on the free Lie algebra $\fn_{d,t}$ for which $\Ker(B)\subseteq C^1(\fn_{d,t})$ and $C^t(\fn_{d,t})$ is not  contained in $\Ker(B)$ and whose morphisms are the metric Lie endomorphisms of
$\fn_{d,t}$ that
respect the kernel of the bilinear forms $\varphi(\Ker(B_1)\subseteq \Ker(B_2)$ for $\varphi$ a morphism between the objects $B_1$ and $B_2$. 
These categories are well defined and they are equivalent. 

\begin{theorem}\cite{BCL} The categories $\mathrm{Sym}_0(d,t)$ and
$\mathrm{NilpQuad}_{d,t}$ are equivalent.
\end{theorem}
This implies that $B_1$ is equivalent to $B_2$ in $\mathrm{Sym}_0(d,t)$ if and only if there exist a metric automorphism
$\theta: (\fn_{d,t},B_1)
\to (\fn_{d,t},B_1)$. So,
the classification of
 quadratic $t$-step nilpotent Lie algebras with
$d$
generators up to isometric isomorphisms
is the classification of objects in the category
$\mathrm{NilpQuad}_{d,t}$ 
up to isomorphism. The group of automorphisms of
$\fn_{d,t}$ acts on the set $Obj(\mathrm{Sym}_0(d,t))$. Moreover the  number of orbits of the action 
is exactly the number of isomorphism types in the classification of
quadratic $t$-step nilpotent Lie algebras of type
$d$ up to isometries.

More examples are shown in the next section.

\section{Particular families}

The aim here is to expose particular families of Lie algebras with ad-invariant metrics. Some of them were obtained by the double extension procedure, the $T^*$-extension process or as twofold extensions. But some results made use of other techniques, which depend on the family and features of ad-invariant metrics. They involved a finer work with Lie theory and constitute the more recent results.

Many examples of Lie algebras with ad-invariant metrics arise when starting the double extension procedure with an abelian Lie algebra $\fd$.

Let $n\in \NN$ and let $(n_1, \hdots, n_t)$ denote a partition of $n$. Consider the partitions satisfying the condition
\begin{equation}\label{(c)}
\mbox{if } i \mbox{ is even then } \# \{j; n_j = i\} \mbox{ is even too. }
\end{equation}

Favre and Santharoubane made use of this condition to determine quadratic nilpotent Lie algebras.

\begin{proposition} \cite{FS}
 The set of isometry classes of regular quadratic Lie
algebras which are extensions of the abelian regular quadratic Lie algebra of
dimension $n$ by nilpotent derivations is in bijection with the set of partitions of
$n$ satisfying Equation (\ref{(c)}).
\end{proposition}

At the dimensions $3, 4, 5, 6$ and $7$ we have to look for the
partitions of $n = 3,4, 5$ satisfying the conditions in Equation (\ref{(c)}), which are as follows:
$$\begin{array}{rcll}
n & = & 3 & (3) \quad (1, 1, 1)\\
n & = & 4 & (3,1) \quad (2,2) \quad (1, 1, 1, 1)\\
n & = & 5 & (5) \quad (3, 1, 1) \quad (2, 2, 1)
\end{array}
$$
so that the indecomposable nilpotent Lie algebras are: $\fn_{2,3}$ the free 3-step nilpotent Lie algebra in two generators, $\fn_{3,2}$ the free 2-step nilpotent Lie algebra in three generators. The other partitions corresponds to trivial extensions of these Lie algebras. See \cite{FS} and considerations below. 

Equip the abelian Lie algebra  $\fd$  with a metric $\la\,,\,\ra$ of signature $(p,q)$. Let $A$ denote a non-trivial skew-symmetric map $A:\fd \to \fd$. Denote by $A(p,q)$ the resulting Lie algebra after the double extension procedure $\frg=\RR t \oplus \fd \oplus \RR z$, where $[t,x] =Ax$ for all $x\in \fd$ and $[x,y]=\la A x, y\ra z$. It is not hard to see that the center of $\fd$ and the commutator are respectively given by

\smallskip

$\fz(\fd)=\Ker \, A \oplus \RR z, \qquad C(\frg)=Im A \oplus \RR z$.

\smallskip

In particular if $\fj$ is an ideal of $\fd$ we would have 
$\fj=\{a A + x + cz: x\in \fd\cap \fj\}$. Thus $A(\fd \cap \fj)\subseteq \fd\cap \fj$. And it is clear that $\RR z\subseteq \fj$ whenever there exist $x,y\in\fj$ such that $\la Ax,y\ra\neq 0$.

%Let  $v\in \fd$ be an eigenvector of $A$ and assume $\la v, v\ra\neq 0$. Take $\fd=\RR v \oplus \fm$ as orthogonal direct sum. Then $\la v, u\ra=0=\la v, Au\ra$ for all $u\in \fm$, which means that $A\fm \subseteq \fm$.  Moreover $\la Av, v\ra=\lambda$ gives  $\lambda=0$. And $[a t + sv + x + uz, v]= a\lambda v - \la Av, x\ra z=0$ for $x\in \fm$;    which implies  that $\RR v $ is a non-degenerate ideal of $\frg$. 

Notice that if $\Ker A\subset \fz(\frg)$, then it  is an ideal of $\frg$. So $\frg$ would be indecomposable if $\frg$ is degenerate.

\begin{proposition} \cite{BK} \label{lor}
 Let $A(p, q)$ be the double extension of an abelian  Lie algebra equipped with a metric of signature (p, q) by
$(\RR, A)$ where $A\in \mathfrak{so}(p,q)$. Then $A(p, q)$ is indecomposable if and only if $\Ker A$ is degenerate. Furthermore $A(p, q)$ is
solvable. 

The corresponding simply connected Lie group with the associated bi-invariant metric is scalar flat and has a 2-step nilpotent Ricci tensor. It is  Ricci-flat if and only if $tr \,A^2=0$, moreover it is flat if and only if $A^2=0$.
\end{proposition}

The Lie algebra $ A(p, q)$ is nilpotent if and only if $A$ is nilpotent. In this case $A(p, q)$ gives rise to  an
Einstein space of scalar curvature 0 (non-flat if $A^2\neq 0$). These nilpotent Lie algebras can be classified using conditions in Proposition \ref{isom} (\cite{FS}). Medina in \cite{Me} get the quadratic Lie algebras of index 1.

\begin{theorem} \label{Me} \cite{Me}
 Each indecomposable Lie algebra equipped with a metric of Lorentzian signature $(1, n-1)$ is isomorphic either to $\mathfrak{sl}(2,\RR)$ or to a solvable Lie algebra which is exactly one of the double extensions
$$A_{\lambda}(0, 2m), \qquad \lambda = (\lambda_1, . . .,\lambda_m), \lambda_1 = 1 \leq \lambda_2 \leq \hdots \leq \lambda_m$$
where $n=2m-2$, $\Lambda$ is the skew-symmetric map given by
$$A_{\lambda} = \left( \begin{matrix}
\Lambda_1  &  & 0 \\
& \ddots & \\
0 & & \Lambda_m \end{matrix} \right) \qquad 
\Lambda_j= \left( \begin{matrix}
0 & -\lambda_j\\
\lambda_j & 0 \end{matrix} \right).$$
\end{theorem}

Baum and Kath study also the double extension procedure and get results for index 2.

\begin{proposition} \label{BK} \cite{BK} If $\frg$ is an indecomposable Lie algebra obtained as a double extension of quadratic Lie algebra and if $\frg$ has  signature $(2,n-2)$
then it is isomorphic to
\begin{enumerate}
\item $\fb$ the Lie algebra $\fb$ of  the Boidol group defined in Example \ref{4dim}, if n = 4,
\item $\fn_{2,3}$ if n = 5, where $\fn_{2,3}$ denotes the free 3-step nilpotent Lie algebra in two generators;
\item $L_{2,\lambda}(1,n-3)$ for $n>5$ even, or
\item $L_{3,\lambda}(1,n-3)$ for $n>5$ odd,
where
$$L_2 :=\left(\begin{matrix} 0 & 1 \\ 1 & 0
\end{matrix} \right) \qquad  
L_3 :=\left(\begin{matrix} 0 & 1 & 0 \\ 1 & 0 & 1 \\
0 & -1 & 0
\end{matrix} \right)
$$
$$
L_{2,\lambda} := \left( \begin{matrix}
L_2 & 0 \\
0 & A_{\lambda} 
\end{matrix} \right) \qquad  L_{3,\lambda} := \left( \begin{matrix}
L_3 & 0 \\
0 & A_{\lambda} \end{matrix} \right)
$$
where $\lambda = (\lambda_1, \lambda_2, . . .,\lambda_r)$ with $0<\lambda_1 <\hdots \lambda_r$ and $A_{\lambda}$ as in Theorem \ref{Me}. 
\end{enumerate}
\end{proposition}

%\begin{remark} We note that Proposition 7.1 in \cite{BK} is true for $n=5$. In fact the free 3-step nilpotent Lie algebra in two generator  $\fn_{2,3}$ admits an ad-invariant metric of signature $(2,3)$ and it has a center of dimension two. 
%\end{remark}

Baum and Kath in \cite{BK} get the indecomposable Lie algebras of dimension $\leq 6$ admitting an ad-invariant metric. The list up to dimension four appears in  Example \ref{4dim}. To get the higher dimensional examples, one  starts 
with the abelian Lie algebra equipped with a fixed metric and one gets the following Lie algebras. 

\begin{enumerate}[(i)]
\item In dimension = 5:
$\fn_{2,3}$ the free 3-step nilpotent Lie algebra in two generators.
\item In dimension = 6: \label{list6}
 \begin{itemize}
\item $\fn_{3,2}$ the free 2-step nilpotent Lie algebra in three generators;
\item $\mathfrak{osc}(\lambda)=span\{e_0, e_1, e_2, e_3, e_4, e_5\}$ where the non-trivial Lie brackets are
$$[e_0, e_1]=e_2, \, [e_0, e_2]=-e_1,\, [e_0, e_4]=\lambda e_5,\, [e_0, e_5]= -\lambda e_4, \, [e_1, e_2]=[e_3, e_4]=e_5$$
\item $L_{2,\lambda}(1,3)$ is a solvable Lie algebra with signature (2,4)
\item $N_k(2,2)$ is a double extension of $\RR^4$ with neutral metric via one of the matrices $N_i$, $i=2,\hdots, 6$ in \cite{BK}.
\end{itemize}
\end{enumerate}

Some of the Lie algebras above can be obtained also as $T^*$-extensions. Recall  the  classification of three dimensional Lie algebras as given e.g. in \cite{Mi}.

\begin{lemma}\label{lie3} Let $\fh$ be a real  Lie algebra of
dimension three spanned by $e_1,e_2,e_3$.
Then it is isomorphic to one in the following list:
\begin{equation}
\begin{array}{rll}
\fh_3 & [e_1,e_2]=e_3 \\
\fr_{3} & [e_1,e_2]=e_2,\, [e_1,e_3]= e_2 + e_3 \\
\fr_{3,\lambda} & [e_1,e_2]=e_2,\, [e_1,e_3]= \lambda e_3 &  |\lambda| \leq 1 \\
\fr_{3,\eta}' & [e_1,e_2]=\eta e_2- e_3,\, [e_1,e_3]= e_2 + \eta e_3 & \eta \geq 0\\
\mathfrak{sl}(2) & [e_1,e_2]=e_3,\, [e_3,e_1]= 2e_1,\, [e_3,e_2]= -2e_2\\
\mathfrak{so}(3) & [e_1,e_2]=e_3,\, [e_3,e_1]= e_2,\, [e_3,e_2]= -e_1
\end{array}
\end{equation}
\end{lemma}

 Let $\ct^* \fh$ denote the cotangent Lie
algebra of a Lie algebra of dimension three as above, spanned by $e_1, e_2, e_3, e_4, e_5, e_6$.  The nonzero
Lie brackets  are listed below:
\begin{equation}  \label{cot3}
\begin{array}{ll}
\ct^*\fh_3: & [e_1,e_2]=e_3, \, [e_1, e_6]=-e_5, [e_2, e_6]=e_4\\
\ct^*\fr_{3}: & [e_1,e_2]=e_2,\, [e_1,e_3]= e_2 + e_3,\, \\
& [e_1, e_5]= -e_5 -e_6,\, [e_1, e_6]=-e_6,\, [e_2, e_5]=e_4,\,
[e_3, e_5]=e_4,\, [e_3,e_6]=e_4 \\
\ct^*\fr_{3,\lambda}: & [e_1,e_2]=e_2,\, [e_1,e_3]= \lambda e_3  \\
  |\lambda| \leq 1& [e_1, e_5]= -e_5,\, [e_1, e_6]=-\lambda e_6,\, [e_2, e_5]=e_4,
\,  [e_3,e_6]=\lambda e_4\\
 \ct^* \fr_{3,\eta}': & [e_1,e_2]=\eta e_2- e_3,\, [e_1,e_3]= e_2 + \eta e_3,
 [e_1, e_5]= -\eta e_5-e_6,\,  \\
 \eta \geq 0 &  [e_1, e_6]= e_5-\eta e_6,\, [e_2,e_5]=\eta e_4,\, [e_2,e_6]= -e_4,
 \, [e_3, e_5]= e_4,\, [e_3,e_6]=\eta e_4\\
\end{array}
\end{equation}

\begin{remark} Notice that these cotangent Lie algebras have ad-invariant metrics of signature (3,3). Obtained by the double extension process one should start with a four dimensional Lie algebra. It is not hard to see that both the oscillator Lie algebra and the other one for the Boidol group do not admit exterior skew-symmetric derivations. In the list above the Lie algebra $\ct^*\fh_3$ corresponds to the  free 2-step nilpotent Lie algebra in three generator $\fn_{3,2}$. The cotangent Lie algebras correspond to the Lie algebras obtained by the double extension procedure, denoted  $N_k(2,2)$ in the list of Baum and Kath (\ref{list6}).
\end{remark}

Recall that %all of them are $T^*$-extensions with $\theta=0$. 
Duong gave all $T^*$-extension of three-dimensional Lie algebras over $\CC$. Except for the abelian situation, the other possibilities correspond to the trivial cocycle  \cite{Du}.

As an application of the twofold extension Kath and Olbrich found quadratic Lie algebras whose respective metrics have index two. 

For this let us denote by 
\begin{itemize}

\item $\mathfrak{osc}(\lambda^1, \lambda^2)$ the Lie algebra which is obtained as the double extension of the oscillator Lie algebra $A_{\lambda^1}(0,2m)=\RR Z_1 \oplus \RR^{2m} \oplus \RR L_1$ by $\RR$ with the map $A_{\lambda^2}$ which acts trivially on $span\{Z_1, L_1\}$ and it acts as $A_{\lambda^2}$ in Theorem \ref{Me} on the subspace $\RR^{2m}\subset \mathfrak{osc}(\lambda^1)$. 

\item $\fd(\lambda^1, \lambda_2)$ the Lie algebra which is $\RR A_0 \oplus \fd$ where $\delta$ is the vector space underlying $\mathfrak{osc}(\lambda^1, \lambda^2)$ where the Lie brackets on $\fd$ are the the same as in $\mathfrak{osc}(\lambda^1, \lambda^2)$ except that $[L_1, L_2]=A_0$, and $[A_0, L_1]=Z_2, \, [A_0, L_2]=-Z_1$.

\end{itemize}

The following result actually improves Proposition \ref{BK}. 
All cases are considered and   the structure of the Lie algebras of index $(2,q)$ is completely described. 

\begin{theorem} \cite{KO} Let $(\frg, \la \cdot , \cdot \ra)$ be an indecomposable metric Lie algebra of signature
$(2, q)$. If $\frg$ is simple, then it is is isomorphic to $\mathfrak{sl}(2,\RR)$ and $\la \cdot, \cdot\ra$ is a multiple of the
Killing form. If $\frg$ is not simple, then the centre $\fz(\frg)$ of $\frg$ is one- or two-dimensional
and we are in one of the following cases.
\begin{enumerate}
\item If $\dim \fz(\frg) = 1$, then $q$ is even and $(\frg, \la \cdot, \cdot \ra)$ is isomorphic to $A_{L_2}(1,1)$ if $q = 2$, or to exactly one of the spaces $A_{L_2, \lambda}(1, q-1)$  if $q \geq 4$, where
$$A_{L_2} = \left( \begin{matrix} 0 &  1 \\ 1 &  0
\end{matrix} \right) \qquad \mbox{ and } \quad A_{L_2, \lambda}= \left( \begin{matrix}
L_2 & 0\\
0 & A_{\lambda}
\end{matrix}
\right)
$$
for $\lambda$ as in Theorem \ref{Me}. 
\item If $\dim \fz(\frg) = 2$ and $\dim \frg$ is even, then $q = 2m + 2$ with $m\geq 3$ and $(\frg, \la \cdot , \cdot \ra)$
is isomorphic to $\mathfrak{osc}(\lambda^1, \lambda^2)$- There is no $j$ such that $\lambda^1_j=\lambda^2_j=0$ and  
 the set $\{(\lambda^1_i, \lambda^2_i)\}$  is not
contained in the union of two 1-dimensional subspaces.
\item If $\dim \fz(\frg) = 2$ and $\dim \frg$ is odd, then $q = 2m+3$ for $m\geq 0$ and $(\frg, \la \cdot ,  \cdot )$ is isomorphic
to $\fd(\lambda^1, \lambda^2)$ as above and there is no index $j \in \{1, \hdots  , m\}$
such that $\lambda^1_j=\lambda^2_j=0$.
\end{enumerate}
\end{theorem}

\smallskip

Following the notion of quadratic extension Kath and Olbrich get the metric Lie algebras of index three. Firstly they prove that if $(\frg,\la\,,\,\ra)$ is a non-simple indecomposable metric Lie algebra of
index 3, then $\frg/j$ 
is isomorphic to one of the Lie algebras $\fr'_{3,0}$, $\fr_{3,-1}$, $\fh_3$, $\mathfrak{sl}(2,\RR)$, $\mathfrak{su}(2)$ or $\RR^k$ for
$k= 1, 2, 3$. Here $\fr_{3,0}'$ denotes the Lie algebra of the Lie group of motions of the euclidean space $\RR^2$ with the canonical metric and $\fr_{3,-1}$ denotes the corresponding Lie algebra if $\RR^{2}$  is endowed with the neutral metric. To get all metric Lie algebras of index three Kath and Olbrich study the orthogonal $\fl$-modules $\fa$. 

\begin{theorem}\cite{KOl}
If $(\frg, \la\,,\,\ra)$ is a simple metric Lie algebra of index 3, then $\frg$  is isomorphic to $\mathfrak{su}(2)$ or $\mathfrak{sl}(3,\RR)$ and $\la\,,\,\ra$ 
is a positive multiple of the Killing form or it is
isomorphic to $\mathfrak{sl}(2,\CC)$ and
$\la\,,\,\ra$ is a non-zero multiple of the Killing form.

If $(\frg, \la\,,\,\ra)$ is a non-simple indecomposable metric Lie algebra of index 3, then $(\frg, \la\,,\,\ra)$ is isomorphic to exactly one Lie algebra $\fd_{\alpha,\gamma}(\fl, \fa, \rho)$ with the data exposed in  Section 7 in \cite{KOl}. See the table at the end of the section. 
\end{theorem}

\smallskip

The question of determining if a given Lie algebra admits an ad-invariant metric is still open in the more general situation. Most features exposed here can be used in general for the negative answer. However there were some attempts to cover this question in particular infinite families.  

\smallskip

 del Barco and this author determined which free nilpotent Lie algebras admit such a metric. The answer was completely given in \cite{BO}. 

\begin{theorem} \label{t1}  Let $\fn_{m,k}$ be the free k-step nilpotent Lie algebra on $m$ generators. Then $\fn_{m,k}$ admits an ad-invariant
metric if and only if $(m,k) = (3, 2)$ or $(m,k) = (2, 3)$.
\end{theorem}

The proof makes use of Hall basis for nilpotent Lie algebras and properties of ad-invariant metrics. The  result was extended to 
free metabelian nilpotent Lie algebras.

\begin{theorem}  \label{t2}  Let $\tilde{\fn}_{m,k}$ be the free metabelian k-step nilpotent Lie algebra on $m$ generators. Then $\tilde{\fn}_{m,k}$ admits an ad-invariant
metric if and only if $(m,k) = (3, 2)$ or $(m,k) = (2, 3)$.
\end{theorem}

Actually the proof of Theorem \ref{t2} reduces to Theorem \ref{t1} in view of the following lemma. Recall that a Lie algebra $\frg$ is called {\em 2-step solvable} if $[C(\frg),C(\frg)]=0$.

\begin{lemma} Let $\frg$ denote a 2-step solvable Lie algebra provided with an ad-invariant metric, then $\frg$ is nilpotent
and at most 3-step.
\end{lemma}

In particular the application of the previous lemma for a free metabelian Lie algebra $\tilde{\fn}_{m,k}$  says that  $k\leq 3$.  And such Lie algebras are free nilpotent.

Nilpotent metric Lie algebras of dimension at most 10 were determined by Kath in \cite{Ka}. In this work the author makes use of theory developped in \cite{KO}  for Lie algebras
with invariant non-degenerate inner product or, equivalently, for simply-connected Lie
groups with a bi-invariant pseudo-Riemannian metric. The author determines all nilpotent Lie algebras
$\fl$ with $\dim \fl'= 2$
which are used in the classification scheme in \cite{KOl}. Thus she obtained all nilpotent
metric Lie
algebras of dimension at most 10. This list includes the representations of some solvable Lie algebras which need more explanations on representation theory. We refer directly to \cite{Ka}. 

Also for nilpotent Lie algebras, making use of the tools in \cite{BCL}, Benito, de la Concepcion and Laiena determine  that over any field
$K$ of characteristic $0$, the set $Obj(\mathrm{Sym}_0(2,t)$ is empty if $t= 2, 4$. And they determine which are these elements for $t=3, 5$. 
See Theorem 5.2 \cite{BCL}.  In particular, over any field
$K$ of characteristic $0$:

\begin{itemize}
\item there are no nilpotent quadratic
Lie algebras of type 2 and nilpotent index 2 or 4.
 
\item The  quadratic nilpotent Lie algebras of type 2 and nilpotent
index 3 are obtained from $\fn_{2,3}$. 
\item The  quadratic nilpotent Lie algebras of type 2 and nilpotent
index 5 are obtained from $\fn_{2,5}$. 
\end{itemize}

\begin{theorem}\cite{BCL} Up to isomorphism, the indecomposable quadratic nilpotent Lie algebras over any algebraically closed field $K$ of  characteristic zero of type 2 and nilindex $\leq$ 5
or of type 3 and nilindex $\leq$  are isomorphic to one of the
following Lie algebras:
\begin{enumerate}[(i)]
\item
The
 free nilpotent Lie algebra
$(\fn_{2,3}, \phi)$ with basis
$\{a_i\}$ $i=1, \hdots 5$ and nonzero products
$[a_1,a_2]=a_3, \, [a_3,a_1] =a_4$ 
and $[a_3,a_2] = a_5$
where
$\phi(a_i,a_j) =\phi(a_j,a_i) = (-1)^{i-1}$ for
$i\leq j$ and $i+j= 6$ and $\phi(a_i,a_j) = 0$
otherwise.

\item The 7-dimensional Lie algebra $(\fn_{2,5,1}, \phi)$ with basis $\{a_i\}_{i=1}^7$ and nonzero Lie brackets $[a_2,a_1]=a_3,\,$ $[a_1,a_3]=a_4$,\, $[a_1,a_4]=a_5$,\, $[a_1,a_5]=a_6$,\, $[a_2,a_5]=a_7$,;

\item The 8-dimensional Lie algebra $(\fn_{2,5,2}, \phi)$ with basis $\{a_i\}_{i=1}^8$ and nonzero Lie brackets$[a_2,a_1]  = a_3$, $[a_3,a_1] = a_4$, $[a_3,a_2] = a_5$, $[a_4,a_1] = a_6$, $[a_6,a_1] = a_7$ and
$[a_6,a_2] = a_8$ where $\phi(a_i,a_j)= (-
1)^i$ for $i\leq j$ and
$i + j = 8$ and $\phi(a_i,a_j)=0$
otherwise.

\item The free nilpotent Lie algebra $\fn_{3,2}=\ct^* \fh_3$ with the canonical neutral metric, where $\fh_3$ denotes the Heisenberg Lie algebra of dimension three.

\item The 8-dimensional Lie algebra $(\fn_{3,3,1}, \phi)$ with basis $\{a_i\}_{i=1}^8$
and nonzero Lie brackets $[a_2,a_1] =a_4$,\, $[a_3,a_1] =a_5$,\, $[a_4, a_1] = a_6$, $[a_4,a_2] = a_7$ 0
and $[a_5,a_1] = a_8$ where $\phi(a_4,a_4) = \phi(a_5, a_5)  = 1$, $\phi(a_1,a_7) = 1 = -\phi(a_2,a_6) =-\phi(a_3,a_8)$ and $\phi(a_i, a_j)=0$ 
otherwise.

\item The 9-dimensional Lie algebra $(\fn_{3,3,1}, \phi)$ with basis $\{a_i\}_{i=1}^9$
and nonzero Lie brackets $[a_2,a_1] =a_4$,\, $[a_3,a_1] =a_5$,\, $[a_3,a_2] =a_6$,\,$[a_4, a_1] = a_7$,\,  $[a_4,a_2] = a_8$,\, $[a_5,a_1] = a_9=-[a_6,a_1]$ where $\phi(a_4,a_4) = \phi(a_5, a_5)  = \phi(a_6,a_6)=1$, $\phi(a_1,a_8) = 1 = -\phi(a_2,a_7) =-\phi(a_3,a_9)$ and $\phi(a_i, a_j)=0$ 
otherwise.
\end{enumerate}
Any non-abelian quadratic Lie algebra of type $\leq 2$ is indecomposable.
\end{theorem}

The following paragraphs are devoted to 2-step nilpotent Lie algebras with ad-invariant metrics. 
Let $\fn$ denote a 2-step nilpotent Lie algebra with an ad-invariant metric. In this case the following trilinear map
\begin{equation}\label{tri}
\psi(x,y,z)=\la [x,y], z\ra
\end{equation} is alternating. Moreover notice that one only need to consider $x,y,z\in \fh-\fz(\fn)$. This is the situation studied by Noui and Revoy \cite{NR}.

The map above (\ref{tri}) is related to the Koszul map. In fact  over a field and for finite-dimensional Lie algebras, the adjoint of the Koszul map is the map $J: (Sym^2\frg)^{\frg} \to \, Alt^3\frg$
mapping an invariant symmetric bilinear form $B$ on $\frg$
to the alternating trilinear form $J_B(x, y, z) = B([x, y], z)$;
the adjoint of the reduced Koszul map is the resulting map $(Sym^2\frg)^{\frg} \to H^3(\frg)$. To say that the reduced Koszul map is zero means that $J_B$ is an exact 3-cocycle for every
$B\in Sym(\frg)^{\frg}$.  Cornulier found the first example of a nilpotent Lie algebra with nonzero Koszul map. This Lie algebra has dimension 12 and it is 7-step nilpotent. 

It turns out that the vanishing condition of the reduced Koszul map is very
common and the semisimple case is quite peculiar. 
Cornulier proved the following:

\begin{theorem} Assume that the ground ring
$R$ is a commutative $\mathbb Q$-algebra (e.g.
a field of characteristic zero).  Let $\frg$ be a Lie algebra graded in a torsion-free abelian group $A$. Then the reduced Koszul map is zero in every nonzero degree $\alpha$.
\end{theorem}

Actually Cornulier gave several situations where the reduced Koszul map is zero. See more references and applications on 
related topics in \cite{Cou}. 
 
The following definition can be read in  \cite{NR}. The {\em corank} of a 2-step nilpotent Lie algebra $\fn$ is the number given by $corank(\fn):= \dim \fz(\fn)-\dim C(\fn)$.

\begin{proposition}
Let $\fn$ denote a real 2-step nilpotent  Lie algebra with an ad-invariant metric. Then $\fn$ is a direct product of an abelian ideal and a 2-step nilpotent ideal of corank zero of neutral signature. 
\end{proposition}

Given the set of 3-forms on a vector space $V$ denoted by $\Lambda^3 V$ one has a natural action of $GL(V)$ given by
$A \cdot \omega (x,y,z)= \omega(Ax, Ay, Az)$ for all $x,y,z\in V$ and $\omega \in \Lambda^3 V$.

On the other hand the 2-step Lie algebras $(\fn,\la\,,\,\ra)$ and $(\fn',\la\,,\,\ra')$ are i-isomorphic (or isometric isomorphic) if there exists an isomorphism of Lie algebras $A : \fn \to \fn'$ which preserves the ad-invariant metric, that is 
$\la x, y\ra=\la Ax, Ay\ra'$ for all $x,y\in \fn$. In \cite{NR} Noui and Revoy  proves the next result.  

Let $\fn$ denote an indecomposable 2-step nilpotent Lie algebra.  Let $V$ denote the vector space $\fn/\fz$. Assume $\fn$ admits an ad-invariant metric $\la \,,\,\ra$. Then Equation (\ref{tri}) defines an alternating  trilinear  form on $V$, by
$$\omega(\bar{x}, \bar{y}, \bar{z}) = \la [x,y], z\ra$$
where $\bar{u}$ denotes the class of $u\in \fn/\fz$. It is not hard to see that $\omega$ is well defined. 

Conversely given an alternating trilinear form $\omega$ on $V$, take the direct sum as vector spaces $\fn:=V \oplus \fz$ where $\fz \simeq V$ as vector spaces. Choose a linear bijection $T$ between $V$ and $\fz$ and define a neutral metric $\la\,,\,\ra$ in $\fn$ so that $V$ and $\fz$ are totally real subspaces.
 Choose an inner product $\la \, \, \ra_V$ on $V$ and via de linear bijection complete a parity between $V$ and $\fz$: 
$$\la x, Ty \ra= \la x, y\ra_V \quad \mbox{ for } x, y\in V.$$
Define a Lie bracket on $\fn$ by $[x,y]=0$ for all $x\in \fz$, $y\in \fn$ and $[u,v]\in \fz$ so that
$$\la [u,v], w\ra = \omega(u,v,w), $$ 
equivalently $\la T^{-1}[u,v], w\ra_V = \omega(u,v,w)$. This depends on the choice of $\la\,,\,\ra_V$ and the bijection $V$. %Every definition to complete this scheme depends on some particular choice. 

It is clear that if $\varphi: \fn \to \fn'$ is an i-isomorphism then 
$\varphi[x,y]= [\varphi x, \varphi y]$ and $\la \varphi x, \varphi y\ra=\la x, \la y$ induces an action on the alternating trilinear forms on $V=\fn/\fz=\fn'/\fz'$. Here we assume that we can identified these vector spaces. The converse is also true.  

\begin{theorem}
There exists a bijection between isomorphism classes of 2-step nilpotent Lie algebras with ad-invariant metrics of corank zero of dimension 2n and equivalence classes of alternating trilinear forms of rank $n$.
\end{theorem}

In \cite{Ov2} one finds an alternative description of 2-step nilpotent Lie algebras with
ad-invariant metrics. This is done by  a
parallelism with the Riemannian situation: 
A 2-step nilpotent Lie group $N$ endowed with a left-invariant Riemannian metric can be completely described by the following data: the two vector spaces equipped with inner products $(\fz, \la\,,\, \ra_{\fz})$ and $(\fv, \la \,, \, \ra_{\fv})$ and a linear map $j:\fz \to \mathfrak{so}(\fv,\la\,,\,\ra_{\fv})$. 

In fact, the Lie algebra $\fn$ of the Riemannian 2-step nilpotent Lie group $N$ is the orthogonal direct sum as vector spaces
$$\fn =\fz \oplus \fv$$
where one takes $\fz$ as the center of $\fn$ and $\fv$ is taken as $\fz^{\perp}$. The map $j$ satisfies 
\begin{equation}\label{j}
\la j(z) x, y\ra_{\fv}=\la z, [x,y]\ra_{\fz}
\end{equation}
 for all $x,y\in \fv, \, z\in \fz$. In the case of an indecomposable Lie algebra $\fn$ endowed with an ad-invariant metric  and corresponding Lie group $N$, it was proved in \cite{Ov2} that $\fn$ can be described by taking $\fz=\fv$ equipped with an inner product, a linear map $j:\fv \to \mathfrak{so}(\fv,\la\,,\,\ra_{\fv})$ satisfying Equation (\ref{j}) above together with $ j(z) x= -j(x) z$ for all $x,z\in \fv$. 

Also concerning families of nilpotent Lie algebras admitting ad-invariant metrics we found the results of del Barco \cite{dB}. They are surprising in the sense that when looking at known families of nilpotent Lie algebras one does not find many quadrable ones. The next result is for quadrable 2-step nilpotent Lie algebras associated to graphs. 

\begin{theorem} \cite{dB} The Lie algebra $\fn_G$ associated with a graph
$G$  admits an ad-invariant metric if and only if $\fn_G$
is isomorphic to a direct sum $\RR^s\oplus \fn_{3,2} \oplus \hdots \oplus \fn_{3,2}$. 
\end{theorem}

Recall the construction of $\fn_G$. Let $G=(
V, E)$ be a finite simple graph where $V$ is the set of vertices and
$E$ is the nonempty set of edges.  An edge given by the unordered pair
$v, w \in V$ is denoted $e=vw$. Let $
U$ be the subspace of $\Lambda^2 V_1$
spanned by $\{v\wedge w: v, w \in V,\,  vw\in E\}$,  where $V_1$ is the subset of vertices in $V$ belonging to at least one edge. 

 The Lie algebra $\fn_G$ associated to
$G$ is $\fn_G= V \oplus  U$
where the Lie bracket is
obtained by extending the following rules by skew-symmetry. If $v, w \in V$ then $[v, w] = v \wedge w$ 
if $vw\in E$ and $0$ otherwise, and $[v, u] = 0$ for all $v\in V$ and
$u\in U$. The dimension of $\fn_G$ is $|V|+|E|$ and it is 2-step nilpotent.

She also studied the family of nilradicals. The tools to get the next result include  Lie theory techniques so as the features mentioned above. 

\begin{theorem} \cite{dB} Let $\fn$ 
be a nilradical of a  parabolic subalgebra associated to a split real form of a
simple
Lie algebra
$\frg$. Then $\fn$ admits ad-invariant metric if and only if $\fn$ is abelian or it is
isomorphic to either $\fn_{3,2}$ or to $\fn_{2,3}$. 
\end{theorem}

In the non-solvable case Benayadi and Elduque found  non-solvable Lie algebras with ad-invariant metrics of dimension $\leq 13$ \cite{BE}. They make use of the theory of classification of finite-dimensional representations of $\fsl(2,\FF)$. See  the next examples which appear in their list. We just write some of them in order to show the techniques. For the rest we refer to the original paper. %Indeed, some of the constructions giving solvable Lie algebras of dimension less or equal than six here could by obtained also by the double extension procedure.  
 
\begin{example} \label{3.9}. Let $V$ be a two-dimensional vector space, endowed with a non-degenerate
skew-symmetric bilinear form $( \cdot | \cdot)$. On the vector space $\fa = \fsl(V) \oplus (V \otimes V)$ consider the symmetric
non-degenerate bilinear form $B_{\fa})$ such that $B_{\fa}(\fsl(V), V \otimes V) = 0$, $B_{\fa} (f, g) = trace(f g)$ and
$B_{\fa}(u_1 \otimes v_1, u_2 \otimes v_2) = (u_1|u_2)(v_1|v_2)$, for $f, g \in \fsl(V)$ and $u_1, u_2, v_1, v_2 \in V$. Consider the vector
space $\fa$ as an abelian Lie algebra. The linear map $\psi : \fsl(V) \to \Dera(\fa, B_{\fa})$ given by $\psi( f )(g) = [ f, g]$, $\psi( f )(u \otimes v) = f (u) \otimes v$,
for any $f, g \in \fsl(V)$ and $u, v \in V$, is a Lie algebra homomorphism. The double extension of $(\fa, B_{\fa})$ by $(\fsl(2,\FF), \psi)$  is an irreducible quadratic Lie algebra, with a Levi subalgebra isomorphic to
$\fsl(2, F)$ and such that, as a module for this subalgebra, it decomposes as the direct sum of three copies
of the adjoint module, and two copies of the natural two-dimensional module. Its dimension is then
13.
\end{example}
\begin{example} \label{3.11} Let $V$ be a two-dimensional vector space as above, endowed with a non-degenerate
skew-symmetric bilinear form $(\cdot | \cdot)$. On the vector space $\fb = V \otimes V$ consider the
symmetric non-degenerate bilinear form $B_{\fb}$ such that $B_{\fb}(u_1 \otimes u_2, v_1 \otimes  v_2) = (u_1|v_1)(u_2|v_2)$ for
any $u_1, u_2, v_1, v_2 \in V$. Consider the vector space $\fb$ as an abelian Lie algebra. The linear map
$\varphi : \fsl(V) \to \Dera(\fb, B_{\fb})$ given by $\varphi(f )(u \otimes v) = f (u) \otimes v$, for any $f\in \fsl(V)$ and $u, v \in V$, is a
Lie algebra homomorphism. The double extension of $(\fb, B_{\fb})$ by $\fsl(V)$ is a quadratic Lie algebra, with
a Levi subalgebra isomorphic to $\fsl(2.F)$  and such that, as a module for this subalgebra, it decomposes
as the direct sum of two copies of the adjoint module, and two copies of the natural two-dimensional
module. Its dimension is then 10.
\end{example}

\begin{example} \label{3.12} Let $V$ and $(\cdot| \cdot)$ as before. Consider the abelian Lie algebra $\fb = V \otimes V$ as
in Example \ref{3.11}. Fix a basis $\{u, v\}$ of $V$ with $(u|v) = 1$, and the one-dimensional Lie algebra
$\FF d$. Let $\varphi : \FF d \to \Der(\fb)$ be the Lie algebra homomorphism such that $\varphi(d)(u_1 \otimes  u) = u_1  \otimes u$,  $\varphi(d)(u_1 \otimes v) = - u_1 \otimes v$, for any $u_1\in V$. Then the double extension $\frg$ of $(\fb, B_{\fb})$ by $\FF d$ is a solvable Lie algebra of dimension 6, endowed with
an invariant scalar product $B_{\frg}$ of signature (3,3). This should appear in the classification above. 

Consider now the Lie algebra homomorphism $\phi : \fsl(V) \to \Der(\frg, B_{\frg})$ given by $\phi(s)(\FF d +
\FF d^*) = 0, \,\phi(s)(u_1 \otimes u_2) = s(u_1) \otimes u_2$ for any $s\in \fsl(V)$ and $u_1, u_2 \in V$. The double extension
of $(\frg, B_{\frg})$ by $(\fsl(V),\phi)$  is an irreducible quadratic Lie algebra, with a Levi subalgebra isomorphic to $\fsl(2,\FF)$
and such that, as a module for this subalgebra, it decomposes as the direct sum of two copies
of the adjoint module, two copies of the natural two-dimensional module, and two copies of the
trivial one-dimensional module. Moreover, the subalgebra formed by the two copies of the trivial
one-dimensional module act diagonally on the sum of the two natural modules.
\end{example}

 \begin{example} \label{3.13} Let $V$ and $(\cdot|\cdot)$ as before. Consider the abelian Lie algebra $\fb = V \otimes V$ as in
Example \ref{3.11}. Set, as before, a basis $\{u, v\}$ of $V$ with $(u|v) = 1$, and the one-dimensional Lie algebra
$\FF d$. Let $\varphi':\FF d \to \Der(\fb)$ be the Lie algebra homomorphism  $\varphi'(d)(u_1 \otimes  u) = u_1  \otimes u$,  $\varphi(d)(u_1 \otimes v) = 0$, for any $u_1\in V$ and take the double extension of $(\fb, B_{\fb})$ by $(\FF d, \varphi')$ which  is a solvable Lie algebra $\frg'$ of dimension 6, endowed with an
invariant scalar product $B_{\frg'}$. 

Consider now the Lie algebra homomorphism  $\phi': \fsl(V) \to  \Der(\frg, B_{\frg})$ given with the same 
formulas as in Example \ref{3.12}. The double extension of $(\frg', B_{\frg'})$ by $(\fsl, \phi')$ is an irreducible quadratic Lie
algebra, with a Levi subalgebra isomorphic to $\fsl(2,\FF)$ and such that, as a module for this subalgebra,
it decomposes as the direct sum of two copies of the adjoint module, two copies of the natural
two-dimensional module, and two copies of the trivial one-dimensional module. Moreover, the
subalgebra formed by the two copies of the trivial one-dimensional module act in a nilpotent way
on the sum of the two natural modules.
\end{example}

\begin{theorem} \cite{BE} 
The complete list, up to isomorphisms, of the non-solvable indecomposable quadratic real Lie algebras $\frg$ with
$\dim \frg \leq 13$ 
is the following:
\begin{enumerate}
\item $\dim \frg = 3$: the simple Lie algebras $\mathfrak{sl}(2, \RR)$ and the real $\mathfrak{su}(2)$,
\item $\dim \frg =6$ : the real simple Lie algebra $\mathfrak{sl}(2,\CC)$ and the trivial $T^*$-extensions $\ct^* \mathfrak{sl}(2,\RR)$ and $\ct^*\mathfrak{su}(2)$,
\item $\dim \frg = 8$: the simple Lie algebras
$\mathfrak{sl}(3,\RR)$, $\mathfrak{su}(3)$ and $\mathfrak{su}(2,1)$
\item $\dim \frg = 9$: the “scalar extensions”
$\mathfrak{sl}(2,\RR)\otimes_{\RR} \RR[x]/(x^3)$ and $\mathfrak{su}(2)\otimes_{\RR} \RR[x]/(x^3)$, 
\item $\dim \frg = 
10$: the simple Lie algebras $\mathfrak{so}(5,\RR)$, $\mathfrak{so}(4,1)$ and $\mathfrak{so}(3,2)$,
and the Lie algebras in
Example \ref{3.11} and Example 4.7 in \cite{BE}
\item $\dim \frg = 11$: the double extensions
$\fd(4)$  and $\breve{\fd}(4)$, and the Lie algebra in Example 3.14, all references in \cite{BE},
\item $\dim \frg  = 12$:
the ``scalar extensions'' $\mathfrak{s}\otimes \RR \mathcal{A}$ with either $\mathfrak{s}=\mathfrak{sl}(2,\RR)$ or $\mathfrak{su}(2)$ and with $\mathcal{A}=\RR[x]/(x^4)$, $\RR[x,y]/(x^2,y^2)$ or $\RR[x,y](x^3,y^3, x^2-y^2)$,
the Lie algebras in
Examples \ref{3.12}, \ref{3.13}, and Examples 4.8, 4.9 in \cite{BE}, and the trivial $T^*$-
extension $T^*_0(\mathfrak{sl}(3,\CC)$,
\item $\dim \frg = 13$:
the double extensions $\fd(6)$ and $\breve{\fd}(6)$ 
and the Lie algebras in Example \ref{3.9} and
Example 4.4 in \cite{BE}.
\end{enumerate}
\end{theorem}

The authors prove that over $\CC$ the
Levi component of a non-solvable irreducible quadratic Lie algebra is $\fsl(2, \CC)$ in most cases, and then
they make use of  the well-known representation theory of $\fsl(,\CC)$. 

%By Medina and Revoy any indecomposable Lie algebra  should be a double extension of a Lie algebra admitting an ad-invariant metric. 

Notice that a real simple Lie algebra of dimension three is isomorphic either to $\fsl(2,\RR)$ or $\mathfrak{su}(2)$. Hence if $\frg$, with dimension $\leq 13$ is a double extension of a Lie algebra $\fd$ by a simple Lie algebra, then 
$$\frg = \fh \oplus \fd \oplus \fh^*\quad \mbox{ with } \fh=\fsl(2,\RR) \mbox{\,\, or \, \, } \fh = \mathfrak{su}(2)\simeq\mathfrak{so}(3)$$

It is clear that $\dim \fd \leq 7$. The list of indecomposable Lie algebras of dimension $\leq 6$ was given a above. That is one could get $\fd$ from this list by the double extension procedure.  But it is also possible to take decomposable Lie algebras to start the double extension procedure to obtain $\frg$. 

For $\fd$ of dimension seven, one possibility is to take $\fd$ abelian but for the other cases one has non-trivial Lie algebras of dimension $7$  as double extensions of lower dimensional Lie algebras. In this case the possibilities for a quadratic five-dimensional solvable Lie algebra are the abelian Lie algebra and the nilpotent Lie algebra $\fn_{2,3}$, which admit  Lie subalgebras of skew-symmetric derivations, see \cite{BO}.  
A similar reasoning should be done for lower dimensional examples. 
Actually any indecomposable Lie algebra of dimension $\leq 6$ is a double extension of a Lie algebra of dimension $\leq 4$.

\section{Some geometry of Lie groups with bi-invariant metrics}

An ad-invariant metric on a Lie algebra gives rise to a bi-invariant metric on the corresponding connected Lie group, that is, a pseudo-Riemannian metric with is invariant by both, left and right translations by elements of the group.

One has the following equivalences \cite{On}. 
\begin{enumerate}[(i)]
\item $\la\,,\,\ra$ is bi-invariant;
\item $\la\,,\,\ra$ is $\Ad(G)$-invariant;
\item the inversion map $g \to g^{-1}$ is an isometry of $G$;
\item $\la [X, Y], Z\ra + \la Y, [X, Z]\ra = 0$ for all $X, Y, Z \in \frg$;
\item $\nabla_X Y= \frac12 [X,Y]$ for all $X,Y\in \frg$ where $\nabla$ denotes the Levi-Civita connection for $\la \,,\,\ra$.
\item the geodesics of $G$ starting at the identity element $e$ are the one-parameter subgroups of $G$.
\end{enumerate}

By making use of this information one gets that the curvature tensor 
is given by 
$$R(X, Y) = - \frac14 \ad([X, Y])\qquad \mbox{ for  } X, Y \in \frg.$$
Thus the Ricci tensor $Ric(X, Y) = tr(Z \to R(Z, X)Y)$ is given by
$$Ric(X, Y) = - \frac14 B(X, Y)$$
where $B$ denotes the Killing form on $\frg$ given by 
$B(X, Y) = tr(\ad(X) \circ \ad(Y))$.

The formula for the curvature tensor says that 2-step nilpotent Lie groups with bi-invariant metrics are flat. 

Since Lie groups provided with bi-invariant metrics are symmetric spaces, its isometry group can be computed with help of the Ambrose-Singer-Cartan Theorem. See also \cite{Mu}. Its not hard to see that if $G$ denotes a Lie group equipped with a bi-invariant metric, then the isometry group $Iso(G)$ is given by
$$Iso(G)= F(G) \cdot L(G)\quad$$
 where $L(G)$ denotes the subgroup of translations on the left and $F(G)$ denotes the subgroup of isometries fixing the identity element.

\begin{lemma} Let $G$ be a simply connected Lie group with a bi-invariant pseudo-Riemannian metric $\la\,,\,\ra$. Then a linear isomorphism $A : \frg \to \frg$ is the differential of some isometry which fixes the identity element if and only if for all $X, Y, Z \in \frg$, the linear map $A$ satisfies the following two
conditions:
\begin{enumerate}[(i)]
\item $\la AX, AY \ra =  \la X, Y \ra$;
\item $A[[X, Y], Z] = [[AX, AY], AZ]$.
\end{enumerate}
\end{lemma}

As a consequence, on a 2-step nilpotent Lie group $N$ provided with a bi-invariant metric,   any isometry fixing the identity element  does not see the algebraic structure of $N$.  The isometry group of the oscillator groups provided with the bi-invariant metric was computed by Bourseau in \cite{Bou}.

\begin{example} Let $N$ denote a 2-step nilpotent Lie group provided with a bi-invariant metric. Let $\fn$ denote its Lie algebra and assume this has corank zero. Then the isometry group is 
$$Iso(N)=O(n,n) \cdot N,$$
  where $n=\dim \fz(\fn)$, $O(n,n)$ denotes the orthogonal group neutral signature $(n,n)$ and as usual $N$ is identified with the subgroup of left-translations $L(N)$. See \cite{Ov2}. 
\end{example}

In \cite{BBM} the authors study  Lie algebras over a field $\mathbb K$ of null characteristic which admit, at the same time, an ad-invarint metric and a symplectic structure. If $\mathbb  K$ is algebraically closed,  such Lie algebra may be constructed as the $T^*$-extension of a nilpotent Lie algebra admitting an invertible derivation and also as the double extension of another  symplectic Lie algebra with an ad-invariant metric by the one-dimensional Lie algebra. Every symplectic quadratic Lie algebra is a special symplectic Manin algebra and the  double extension procedure applies. See details in the mentioned paper.

\

Acknowledgments: The author thanks the referee for the additional  bibliography  and the suggestions to improve the presentation of this work. 
%A complete survey should  ask for more time and pages than these. This is a necessary work for the future.

G. Ovando also thanks the organizers for  the invitation to participate in the Conference in honour of Sergio Console, Torino february of 2015. 

\smallskip

{\em  Este trabajo est\'a enteramente dedicado a la memoria de Sergio Console, como  colega y amigo}. 

%Use the same for lemma, definition, \ldots.

%Now let's go on the other page to check the look (different from the first.)
%\newpage
%and more a new one

% \newpage

\bigskip

\begin{flushleft}

{\bf AMS Subject Classification: 22E25, 22E60, 53B30}\\[2ex]

% Write more than one author separately if they have different 
% affiliations, otherwise write the names on the same line, separeted 
% by commas.
%
Gabriela P. Ovando,\\
Departamento de Matem\'atica, ECEN-FCEIA, Universidad Nacional de Rosario\\
Pellegrini 250, 2000 Rosario, ARGENTINA\\
e-mail: \texttt{gabriela@fceia.unr.edu.ar}\\[2ex]

% leave it blanck, you dont know these infos yet
%
%\textit{Lavoro pervenuto in redazione il MM.GG.AAAA.}

\end{flushleft}

\end{document}